\documentclass[12pt]{amsart}
\usepackage{hyperref,amssymb}
\usepackage{bbm}
\usepackage{textcmds}

\theoremstyle{plain}
\newtheorem{theorem}{Theorem}
\newtheorem{proposition}[theorem]{Proposition}
\newtheorem{corollary}[theorem]{Corollary}
\newtheorem{lemma}[theorem]{Lemma}

\theoremstyle{definition}
\newtheorem{remark}[theorem]{Remark}

\newtheorem{definition}[theorem]{Definition}
 \newtheorem{case}{Case}

\newcommand{\goesuo}{\xrightarrow{\mathrm{uo}}}	

\DeclareSymbolFont{bbold}{U}{bbold}{m}{n}
\DeclareSymbolFontAlphabet{\mathbbold}{bbold}

\DeclareMathOperator{\supp}{supp}

\begin{document}

\title{Discrete stopping times in the lattice of continuous functions}

\author{Achintya Raya Polavarapu}
\email{polavara@ualberta.ca}

\address{Department of Mathematical and Statistical Sciences,
  University of Alberta, Edmonton, AB, T6G\,2G1, Canada.}



\date{\today}

\begin{abstract}
A functional calculus for an order complete vector lattice $\mathcal{E}$ was developed by Grobler in \cite{MR3151817} using the Daniell integral. We show that if one represents the universal completion of $\mathcal{E}$ as $C^\infty(K)$, then the Daniell functional calculus for continuous functions is exactly the pointwise composition of functions in $C^\infty(K)$. This representation allows an easy deduction of the various properties of the functional calculus. Afterwards, we study discrete stopping times and stopped processes in $C^\infty(K)$. We obtain a representation that is analogous to what is expected in probability theory.
\end{abstract}

\maketitle

\section{Introduction and Preliminaries}
In the first half of this article, we prove some results regarding the functional calculus in vector lattices. Specifically, the main result shows that when represented in $C^\infty(K)$, the Daniell functional calculus introduced by Grobler in \cite{MR3151817} is the pointwise composition of functions for continuous functions. In the second half, we investigate discrete stopping times and their representation in $C^\infty(K)$. In Section \ref{discrete_stopping_times}, we show that discrete stopping times in vector lattices correspond to a certain class of elements in the sup-completion. Moreover, when elements of the vector lattice are viewed as functions in $C^\infty(K)$, the discrete stopping times and stopped processes have a representation similar to what is expected in classical probability. This allows an easy deduction of various properties of discrete stopping times in vector lattices. Furthermore, we provide a generalization of Début theorem to the abstract setting. For more information about stopping times and hitting times in the classical probability theory, we refer the reader to \cite{MR3642325}. 

\subsection{\texorpdfstring{$C^\infty(K)$}{TEXT} representation}
We refer the reader to \cite{MR2011364, MR2262133} for the theory of vector lattices. Given a Hausdorff topological space $K$, we will write $C(K)$ for the set of all continuous functions from $K$ to $\mathbb{R}$. The space $C(K)$ is a vector lattice under pointwise operations, and if $K$ is an extremally disconnected Hausdorff topological space, then $C(K)$ is order (or Dedekind) complete. A subset $A$ of a Hausdorff topological space is nowhere dense if $(\overline{A})^\circ = \emptyset$. This also implies that $\partial A$ is a nowhere dense set for every closed set $A$. A set $A$ is meagre if it is a union of a sequence of nowhere dense sets, and conversely, a set is co-meagre if its complement is meagre. $K$ is a Baire space if countable unions of closed sets with empty interior also have empty interior. Every compact Hausdorff space is a Baire space.

Given that $K$ is an extremally disconnected compact Hausdorff topological space, we write $C^\infty(K)$ for the vector lattice of all continuous functions from $K$ to $[-\infty,\infty]$ that are finite almost everywhere (a.e.), that is, except on a nowhere dense set. Two functions $f$ and $g$ in $C^\infty(K)$ are equal, provided that the set of points where their values differ is a nowhere dense set. Scalar multiplication, addition and lattice operations on $C^\infty(K)$ are defined pointwise a.e. It should be noted that $C^\infty(K)$ is an f-algebra with product defined pointwise a.e. Maeda-Ogasawara Theorem states that every Archimedean vector lattice $\mathcal{E}$ may be represented as an order dense sublattice of $C^\infty(K)$, where $K$ is an extremally disconnected compact Hausdorff topological space. Throughout the paper, we fix an order complete vector lattice $\mathcal{E}$ with a weak unit $E$ and a Maeda-Ogasawara representation of $\mathcal{E}$ as an order dense ideal in $C^\infty(K)$, where $K$ is the Stone space of $\mathcal{E}$, with $E$ corresponding to $\mathbbm{1}$. 

We shall also use the concept of the sup-completion of $\mathcal{E}$ as introduced by \cite{MR653635} which we will denote by $\mathcal{E}^s$. We recall [\cite{PT2023}, Theorem 1] that states that the sup-completion of an order complete vector lattice $\mathcal{E}$ is $\mathcal{E}^s=\{f\in C(K,\overline{\mathbb{R}}):f\geq g \text{ for some } g\in \mathcal{E}\}$. We refer the reader to \cite{PT2023} for more properties regarding the representation of $\mathcal{E}^s$. Every element of the sup-completion can be decomposed into finite and infinite parts \cite{MR4315562, PT2023}. In terms of the above representation, this decomposition is as follows: the finite part $x$ of $u$ is defined as $x=P_U u$ where $U$ is the closure of  $\{u<\infty\}$. The infinite part $w$ of $u$ is defined as $P_{U^c}u$ and is the function that vanishes on $U$ and takes the value $\infty$ on $U^c$. Corollary 7 in \cite{MR3906369, PT2023} also gives us that $\mathcal{E}_+^u\subseteq \mathcal{E}_+^s$.

\subsection{Conditional expectation operator}
The motivation for the following definitions and theory can be found in \cite{MR2093170, MR2122234}. A conditional expectation $\mathbb{F}$ defined on $\mathcal{E}$ is an order continuous strictly positive linear projection whose range $\mathcal{R}(\mathbb{F})$ is an order complete vector lattice of $\mathcal{E}$ with $\mathbb{F}E = E$. Conditional expectation operators satisfy an averaging property [\cite{MR2122234}, Theorem 5.3] in the sense that if $f \in \mathcal{R}(\mathbb{F})$ and $g \in \mathcal{E}$ with $fg \in\mathcal{E}$ then $\mathbb{F}(fg)=f\mathbb{F}(g)$, where the product of two elements is done in $C^\infty(K)$. We note that the range of a conditional expectation operator is a regular sublattice of $\mathcal{E}$. 

Let $J\subset \mathbb{R}_+$. A filtration on $\mathcal{E}$ is a family $(\mathbb{F}_t)_{t\in J}$ of conditional expectations satisfying $\mathbb{F}_s=\mathbb{F}_s \mathbb{F}_t=\mathbb{F}_t \mathbb{F}_s$ for all $s\leq t$. We denote the range of $\mathbb{F}_t$ by $\mathcal{F}_t$. A stochastic process in $\mathcal{E}$ is a family $(X_t)_{t\in J}$ where $X_t\in\mathcal{E}$, with $J\subset \mathbb{R}^+$ an interval. The stochastic process $(X_t)_{t\in J}$ is right continuous if $\text{o-}\lim_{s\downarrow t} X_s = X_t$. The stochastic process $(X_t)_{t\in J}$ is adapted to the filtration if $X_t\in \mathcal{F}_t$ for all $t\in J$. If $(X_t)$ is a stochastic process adapted to $(\mathbb{F}_t)$, we call $(X_t, \mathbb{F}_t)$ a super-martingale (respectively sub-martingale) if $\mathbb{F}_t(X_s) \leq X_t$ (respectively $\mathbb{F}_t(X_s) \geq X_t$) for all $t \leq s$. The process is called a martingale if it is a sub-martingale and super-martingale.

\section{Functional Calculus}\label{functional_calculus}
The functional calculus developed by Grobler plays a major role in the theory of measure-free probability in vector lattices \cite{MR3573115,MR3777886,MR3111875,azouzi2022generating}. In the following subsection, we briefly outline the construction of the functional calculus using the Daniell integral, as shown in \cite{MR3151817}. 

\subsection{Daniell Integral}
Denote by $F(\mathbb{R})$ the algebra consisting of all finite unions of disjoint left open right closed intervals $(a, b], (a,\infty)$ and $(\infty, b]$ with $a, b \in \mathbb{R}$. Let $\mathbb{L}$ be the vector lattice of real valued functions of the form:
\begin{equation}\label{piecewise_function_definition}
f = \sum_{i=1}^{n} a_i \mathbbm{1}_{S_i}, S_i\in F(\mathbb{R})
\end{equation}
where $(S_i)_{i=1}^n$ is a partition of $\mathbb{R}$. The order relation of $\mathbb{L}$ is defined by $f\leq g$ if $f(t)\leq g(t)$ for every $t\in\mathbb{R}$. 

\begin{definition}
A positive linear function $I : \mathbb{L} \to \mathcal{E}$ is called an $\mathcal{E}-$valued Daniell integral on $\mathbb{L}$ whenever, for every sequence $(f_n)$ in $\mathbb{L}$ that satisfies $f_n(t)\downarrow 0$ for every $t\in\mathbb{R}$, it follows that $I(f_n)\downarrow 0$.
\end{definition}
We note that $I$ need not be an order continuous operator. Consider $I:\mathbb{L}\to\mathbb{R}$ where $I(f)=f(0)$. Then $I$ is a $\mathbb{R}-$valued Daniell integral and consider the sequence $f_n=\mathbbm{1}_{(-\frac{1}{n},0]}$. Then $f_n\downarrow 0$ in $\mathbb{L}$ but $I(f_n)\not\to 0$.\\

We shall detail the construction of a specific Daniell integral developed by Grobler that shall be useful in defining a functional calculus on $\mathcal{E}$. For $Y\in\mathcal{E}$ we denote $P_Y$ to be the band projection associated with the band generated by $Y$. Fix $X\in\mathcal{E}$, then the right continuous spectral system of $X$ is the increasing right-continuous stochastic process $A=(A_t)_{t\in \mathbb{R}}$ where $A_t = E - P_{(X-tE)^+}E$. We denote by $A_\infty$ and $A_{-\infty}$ respectively, the supremum and infimum of the process. To define the $\mathcal{E}-$valued Daniell integral, we define a vector lattice measure $\mu_A$ with respect to $(A_t)_{t\in \mathbb{R}}$ as follows:
\begin{itemize}
    \item $\mu_A(a, b] = A_b - A_a$ where $(a,b]\in F(\mathbb{R})$.
    \item  For any finite disjoint union of half open intervals, we have $\mu_A(\cup_{i=1}^{n}I_i) = \sum_{i=1}^n \mu_A(I_i)$
\end{itemize}

Then $\mu_A$ is a countably additive $\mathcal{E}-$valued measure on $F(\mathbb{R})$ as shown in [Lemma 3.7, \cite{MR3151817}] and the functional calculus is defined for elements of $\mathbb{L}$ as: 
$$ I(f) = \sum_{i=1}^{n}a_i \mu_A(S_i), \text{ where } f\in\mathbb{L}\text{ as in } (\ref{piecewise_function_definition})$$
Define $\mathbb{L}^\uparrow := \{f:\mathbb{R}\to \mathbb{\overline{R}} : \exists (f_n)_{n\in\mathbb{N}}\subset \mathbb{L},\text{ such that } f_n(t)\uparrow f(t) \text{ for every } t\in\mathbb{R}\}$. Then the integral $I:\mathbb{L}\to\mathcal{E}$ can be extended to an integral from $\mathbb{L}^\uparrow$ to $\mathcal{E}^s$ as follows: for $f\in\mathbb{L}^\uparrow$, we define $I(f)=\sup_n I(f_n)$ where $(f_n)$ is a sequence in $\mathbb{L}$ such that $f_n\uparrow f$. Then [Lemma 3.2, \cite{MR3151817}] states that this extension is well-defined. The extension also satisfies the following properties.

\begin{lemma}[Lemma 3.4, \cite{MR3151817}]\label{extension_properties}
 The extension of $I$ to $\mathbb{L}^\uparrow$ is well-defined and satisfies the following the properties. 
\begin{itemize}
    \item If $f, g \in\mathbb{L}^\uparrow$ and $f \leq g$, then $I(f)\leq I(g)$;
    \item If $f\in\mathbb{L}^\uparrow$ and $0 \leq c <\infty$, then $cf \in\mathbb{L}^\uparrow$ and $I(cf) = cI(f)$;
    \item If $f,g \in\mathbb{L}^\uparrow$, then $f+g\in\mathbb{L}^\uparrow$ and $I(f+g)=I(f)+I(g)$;
\item  If $(f_n)_{n\in\mathbb{N}}$ is a sequence in $\mathbb{L}^\uparrow$ and $f_n(t)\uparrow f(t)$ for every $t$ , then $f\in\mathbb{L}^\uparrow$ and $I(f_n)\uparrow I(f)$.
\end{itemize}
\end{lemma}

\subsection{Representation of functional calculus}
For $A\subseteq\mathbb{\overline{R}}$ and $\tau:K\to\overline{\mathbb{R}}$, we shall use the notation $\{\tau\in A\}$ to denote the set $\{\omega\in K: \tau(\omega)\in A\}$. 
We will need the following criterion for the convergence of sequences in a sup-completion. This essentially follows from the results in Section 3 of \cite{MR4366912}. However they are not stated there in the form of the sup-completion, and thus we repeat certain arguments in the next two lemmas.

\begin{lemma}\label{topological_sup_convergence}
For $G \subseteq \mathcal{E}^s_+, \sup G = \infty\mathbbm{1}$ iff for every non-empty open set $U$ and every $n\in\mathbb{R}_+$ there exists a non-empty open set $V \subseteq U$ and $g \in G$ with $g(t) > n$ for all $t\in V$.
\end{lemma}
\begin{proof}
Suppose that $\sup G\neq\infty\mathbbm{1}$. Then there exists $f \in\mathcal{E}^s_+$ with $G\leq f<\infty\mathbbm{1}$. Hence, there exists a point $t\in K$ such that $f(t)<\infty$. Since $K$ is a totally disconnected space, we can find a clopen subset $U\subset K$ containing $t$ such that $f(U)<\infty$. Hence, by compactness there exists $n\in\mathbb{N}$ such that $f$ is less than $n$ on $U$. It follows that every $g\in G$ is less than $n$ on every open subset $V$ of $U$.

Suppose now that there exists an open non-empty set $U$ and $n\in\mathbb{R}_+$ such that for every non-empty open subset $V\subseteq U$ and $g\in G$ there exists $t\in V$ such that $g(t)\leq n$. So for a given $g\in G$, we have that $\{g\leq n\}$ is a closed set and it intersects every open subset of $U$. This implies that $\{g\leq n\}$ contains $U$. Thus, we have that $\cap_{g\in G}\{g\leq n\}\supseteq U$ and clearly, $\sup G\leq \infty\mathbbm{1}_{\overline{U}^c}+n\mathbbm{1}_{\overline{U}}<\infty\mathbbm{1}$.
\end{proof}
The above topological property allows us to deduce properties regarding the convergence in sup-completion.

\begin{lemma}\label{convergence_criterion}
Let $X\in\mathcal{E}^s_+$ and $(X_n)_{n\in\mathbb{N}}\subset\mathcal{E}^s_+$ such that $X_n\uparrow X$. Then there exists a co-meagre set $M\subseteq K$ such that $X_n(t)\uparrow X(t)$ for every $t\in M$.
\end{lemma}
\begin{proof}
Let us denote $U=\overline{\{X<\infty\}}$, which gives that $U$ is a clopen set and $X\mathbbm{1}_U\in C^\infty(K)$. Let us split all the elements as follows: $X_n = X_n\mathbbm{1}_U+X_n\mathbbm{1}_{U^c}$. We have that $X_n\mathbbm{1}_U\uparrow X\mathbbm{1}_U$ and hence by [Remark 4.1, \cite{MR4366912}], the convergence here is pointwise on a co-meagre subset of $U$. It remains to show this for $X\mathbbm{1}_{U^c}$. Note that for this function we have that the value is $\infty$ everywhere. 

Let $G=(X_n\mathbbm{1}_{U^c})_{n\in\mathbb{N}}\subset \mathcal{E}_+^s$. For $n \in\mathbb{N}$, put $V_n =\bigcup_{g\in G} \{g > n\}$. Clearly, $V_n$ is open and $V_n\subseteq U^c$. For every non-empty open set $W$, Lemma \ref{topological_sup_convergence} gives that there exists $t \in W$ and $g \in G$ such that $g(t)>n$. Hence $t \in V_n$. That is, for every open set $W$, there exists a point $t\in W$ such that $t\in V_n$. Hence, $V_n$ is a a dense subset of $U^c$. Let $V := \bigcap_{n=1}^\infty V_n$, then $V$ is co-meagre subset of $U^c$. Let $t \in V$, then for every $n \in \mathbb{N}$ we have $t \in V_n$, and hence $\sup_{g\in G} g(t) > n$. It follows that $\sup_{g\in G} g(t) = \infty$. Therefore, upon combining the co-meagre subsets of $U$ and $U^c$, we obtain the desired co-meagre subset of $K$. 
\end{proof}

Given a continuous function $f\in C(\mathbb{R})$ and $X\in C^\infty(K)$, let $U$ be the set on which $X$ is finite. Clearly, $U$ is open and dense, and the composition of $f$ and $X$ is defined, finite, and continuous on $U$. Then the composition extends uniquely to a function in $C^\infty(K)$. We shall denote this extended function by $f\circ X$. Also note that by the Maeda-Ogasawara theory, there is a one-to-one correspondence between clopen subsets of $K$ and bands in $\mathcal{E}$. For an element $Y\in\mathcal{E}$, the clopen set corresponding to the band generated by $Y$ is the set $\overline{\{Y\neq 0\}}$. Hence the corresponding band projection of $E$ is $P_Y E = \mathbbm{1}_{\overline{\{Y\neq 0\}}}$. 

\begin{lemma}
\label{piecewise_functional_representation}
Let $f\in\mathbb{L}$ and  $\{-\infty=\gamma_0<\gamma_1<\dots<\gamma_{n}<\infty\}$ be a partition of $\mathbb{R}$ such that $f=a_\infty\mathbbm{1}_{(\gamma_n,\infty)}+\sum_{i=1}^n a_i \mathbbm{1}_{(\gamma_{i-1},\gamma_i]}$. Then the Daniell integral of $f$ is 
$$I(f)=a_\infty\mathbbm{1}_{K\setminus\overline{\{X>b\}}}+\sum_{i=1}^n a_i\mathbbm{1}_{\overline{\{X>\gamma_{i-1}\}}\setminus\overline{\{X>\gamma_i\}}}$$
\end{lemma}
\begin{proof}
Note that since the weak unit $E$ is fixed, it corresponds to $\mathbbm{1}$ in the stone representation. Now, we will first prove the result in the case when $f = \mathbbm{1}_S$ where $S = (a,b]\in F(\mathbb{R})$. So,
$$ I(f) = \mu_A(S) = A_b - A_a = (E-P_{(X-bE)^+}E)-(E-P_{(X-aE)^+}E)$$
$$  =  P_{(X-aE)^+}E - P_{(X-bE)^+}E  $$
For $(X-aE)^+$, we have the corresponding band projection $P_{(X-aE)^+}E = \mathbbm{1}_{\overline{\{X>a\}}}$. Hence,
$$I(f)=P_{(X-aE)^+}E - P_{(X-bE)^+}E = \mathbbm{1}_{\overline{\{X>a\}}\setminus\overline{\{X>b\}}}$$
It is easy to see that $\overline{\{X>a\}}\setminus\overline{\{X>b\}}$ is clopen. If $S=(a,\infty)$, the above argument can be adapted to give us
$$ I(f) = \mu_A(S) = A_\infty - A_a = \sup_{b\in \mathbb{R}}(E-P_{(X-bE)^+}E)-(E-P_{(X-aE)^+}E) $$
$$=  \sup_{b\in \mathbb{R}}(P_{(X-aE)^+}E - P_{(X-bE)^+}E)$$
$$=\sup_{b\in\mathbb{R}}\mathbbm{1}_{\overline{\{X>a\}}\setminus\overline{\{X>b\}}}$$
$$=\mathbbm{1}_{\overline{\{X>a\}}}$$
Similarly, when $S=(-\infty, b]$, we have $I(f)=\mathbbm{1}_{K\setminus\overline{\{X>b\}}}$. Therefore when $f\in\mathbb{L}$ is a piece-wise constant function of the form $f=a_\infty\mathbbm{1}_{(\gamma_n,\infty)}+\sum_{i=1}^n a_i \mathbbm{1}_{(\gamma_{i-1},\gamma_i]}$, we have:
$$I(f)=a_\infty\mathbbm{1}_{K\setminus\overline{\{X>b\}}}+\sum_{i=1}^n a_i\mathbbm{1}_{\overline{\{X>\gamma_{i-1}\}}\setminus\overline{\{X>\gamma_i\}}}$$
\end{proof}

\begin{lemma}\label{pointwise_piecewise_equal}
Let $f\in \mathbb{L}$ such that $f$ has bounded support. Then for an element $X\in \mathcal{E}$, there exists an open dense set $H\subseteq K$ such that $I(f)(\omega)=f(X(\omega))$ for every $\omega\in H$.
\end{lemma}
\begin{proof}
Let us suppose that $(a,b]$ is the support of $f$. Then there exists a partition $\{a= \alpha_0 < \alpha_1 < \dots < \alpha_{k}=b \}$ of $(a,b]$ such that $f = \sum_{j=1}^{k} a_j \mathbbm{1}_{(\alpha_{j-1}, \alpha_j]}$. Then Lemma \ref{piecewise_functional_representation} gives us
$$I(f)=\sum_{j=1}^k a_j\mathbbm{1}_{\overline{\{X>\alpha_{j-1}\}}\setminus\overline{\{X>\alpha_j\}}}$$
Now set $H_{j}=\overline{\{X>\alpha_j\}}\cap\{X\leq \alpha_j\}$ and $N=\{X=\infty\}$. Now, $\{X>\alpha_j\}$ is an open set and $\{X\leq \alpha_j\}$ is a closed set of $K$ satisfying $\{X>\alpha_j\}\cap\{X\leq \alpha_j\}=\emptyset$ which implies that $H_{j}$ is a closed nowhere dense set. Since each of the sets are closed and nowhere dense, upon setting $M=\biggl(\bigcup_{j=1}^{k}H_{j}\biggr)\cup N$ we have that $M$ is a closed nowhere dense set and $H:=K\setminus M$ is an open dense set. Let $\omega\in H$. Then we consider two separate cases:
\begin{case}
If $X(\omega)\notin(a,b]$: Then $f(X(\omega))=0=I(f)(\omega)$.
\end{case}
\begin{case}
If $X(\omega)\in(a,b]$: Then there exists $j_0$ such that $X(\omega)\in(\alpha_{j_0-1},\alpha_{j_0}]$: Then we have that $f(X(\omega))=a_{j_0}$. However, we have that $\omega\in \overline{\{X>\alpha_{j_0-1}\}}$ and $\omega\in \{X\leq \alpha_{j_0}\}$. By definition of the set $H$, $\omega\notin\overline{\{X>\alpha_{j_0}\}}$ and thus, $I(f)(\omega)=a_{j_0}$. Hence, $I(f)(\omega)=f(X(\omega))$.
\end{case}
Therefore, we have that $I(f)(\omega)=f(X(\omega))$ for every $\omega\in H$. This implies that $f\circ X=I(f)\in\mathcal{E}$ when $f\in\mathbb{L}$.
\end{proof}

\begin{lemma}\label{positive_bounded_support}
Let $f\in C(\mathbb{R})$ be a positive continuous function with bounded support. Then for an element $X\in \mathcal{E}$, we have $I(f)=f\circ X\in \mathcal{E}_+$.
\end{lemma}
\begin{proof}
Because $f$ has bounded support implies $f\leq \lambda \mathbbm{1}_{(a,b]}$ for some $a,b\in\mathbb{R}$. Then $I(f)\leq \lambda I(\mathbbm{1}_{(a,b]})\in\mathcal{E}$. As $\mathcal{E}$ is order complete, we have $I(f)\in\mathcal{E}_+$. Since $f$ has bounded support, $f$ is an uniformly continuous function and there exists a sequence $(f_k)$ in $\mathbb{L}$ such that $f_k\uparrow \leq f$ and $(f_k)$ converges to $f$ uniformly. WLOG, passing to a subsequence, we have $0\leq f-f_k\leq \frac{1}{k}\mathbbm{1}_{\mathbb{R}}$. By Lemma \ref{pointwise_piecewise_equal}, $f_k\circ X\in\mathcal{E}$ and thus, $f\circ X-f_k\circ X=(f-f_k)\circ X\leq (\frac{1}{k}\mathbbm{1}_{\mathbb{R}})\circ X=\frac{1}{k}\mathbbm{1}_K$. Therefore, $f_k\circ X$ converges to $f\circ X$ relatively uniformly in $\mathcal{E}$. Similarly, $I(f)-I(f_k)=I(f-f_k)\leq I(\frac{1}{k}\mathbbm{1}_{\mathbb{R}})=\frac{1}{k}\mathbbm{1}_K$. So $I(f_k)$ converges to $I(f)$ relatively uniformly in $\mathcal{E}$. Since $I(f_k)=f_k\circ X$ by Lemma \ref{pointwise_piecewise_equal}, passing to the limit gives, $I(f)=f\circ X$.

\end{proof}

In view of the above theorem, we will show that for continuous functions the Daniell functional calculus satisfies a nice representation. The idea for the proof of the theorem can be found in Groblers' paper. Let $\mathbb{L}^\uparrow_0$ be the set of all real valued positive functions in $\mathbb{L}^\uparrow$ and
$$\mathbb{L}_u = \{f:f=g-h \text{ such that } g,h\in\mathbb{L}^\uparrow_0 \text{ and } I(g),I(h)\in\mathcal{E}^u\}$$
Then analogous to the proof of [Proposition 3.5, \cite{MR3151817}] we can show that $\mathbb{L}_u$ is a vector space and $I$ has a well-defined extension to $\mathbb{L}_u$ by defining for $f=g-h\in\mathbb{L}_u$, $I(f) = I(g)-I(h)$. The proof also shows that the extension is positive and linear on $\mathbb{L}_u$.

\begin{theorem}\label{pointwise_final_theorem}
Let $f\in C(\mathbb{R})$. Then for an element $X\in \mathcal{E}$, we have $I(f)=f\circ X$.
\end{theorem}
\begin{proof}
Let $f\in C(\mathbb{R})$ be a positive continuous function. Then we can find a sequence $(f_n)\subset C(\mathbb{R})$ of positive continuous functions with bounded support as follows: $f_n(t)=f(t)$ if $t\in [-n, n]$, and $f_n(t)=0$ if $t\notin [-n-1, n+1]$. Then $f_n \uparrow f$ and combining Lemma \ref{extension_properties} and Lemma \ref{convergence_criterion} gives that $I(f_n)\uparrow I(f)$ pointwise on a co-meagre subset of $K$. However, we also have that $f_n\circ X(\omega)\uparrow f\circ X(\omega)$ where $\omega$ belongs to $\{X<\infty\}$, an open dense set. Since Lemma \ref{positive_bounded_support} gives that $I(f_n)=f_n\circ X$, we can conclude that $I(f)=f\circ X$ on a co-meagre set and thus on a dense set by the Baire category theorem. As $f$ is a positive continuous function, implies that $f\circ X\in\mathcal{E}_+^u$ and therefore $f\circ X\in \mathcal{E}_+^s$. Since $I(f)$ and $f\circ X$ are continuous functions equal on a dense subset of $K$, we have $I(f)=f\circ X$ everywhere on $K$. However, $f\circ X\in \mathcal{E}^u$ implies that $I(f)=f\circ X\in \mathcal{E}^u$. Hence, $C(\mathbb{R})_+\subseteq \mathbb{L}_u$ and thus, $C(\mathbb{R})\subseteq \mathbb{L}_u$. So given $f\in C(\mathbb{R})$, applying the preceding argument to $f^+$ and $f^-$, we get
$$I(f)=I(f^+)-I(f^-)=f^+\circ X-f^-\circ X=f\circ X$$
\end{proof}

The following are some simple corollaries resulting from the above theorem.
\begin{corollary}
Let $X\in \mathcal{E}$ and $f\in C(\mathbb{R})$. Then $I(f) = f(X)\in \mathcal{E}^u$.
\end{corollary}
\begin{corollary}\label{wedge_vee_composition}
Let $I:\mathbb{L}_u\to \mathcal{E}^u$ and $f,g\in C(\mathbb{R})$. Then $I(f\vee g)=I(f)\vee I(g), I(f\wedge g)=I(f)\wedge I(g)$ and $I(|f|)=|I(f)|$. That is, the restriction of $I$ to $C(\mathbb{R})$ is a lattice homomorphism.
\end{corollary}

The following proposition is a stronger statement than [Proposition 4.6, \cite{MR3151817}] and [Proposition 2.8, \cite{MR3573115}]. Recall the following convergence criterion for elements of $C^\infty(K)$ from [Theorem 3.7, \cite{MR4366912}] that states that $x_n\goesuo x$ if and only if $x_n$ converges to $x$ pointwise on a co-meagre set. 

\begin{proposition}\label{order_converegence_of_integral}
Let $f:\mathbb{R}\to\mathbb{R}$ be a continuous function. If $x_n\goesuo x$ in $\mathcal{E}$ then $f(x_n)\goesuo f(x)$ in $\mathcal{E}^u$. 
\end{proposition}
\begin{proof}
By the convergence criterion, $x_n(\omega)\to x(\omega)$ for every $\omega\in H$ where $H\subseteq K$ is co-meagre set. By continuity of $f$ this implies that $f[x_n(\omega)]\to f[x(\omega)]$ for $\omega\in H$. Therefore, $f(x_n)\goesuo f(x)$.
\end{proof}

\subsection{Multivariate functional calculus}
Since the Daniell functional calculus for continuous functions corresponds to the pointwise composition of functions, we can extend this to the concept of multivariate continuous functions. Given $f\in C(\mathbb{R}^n, \mathbb{R})$ and  $\mathbf{X}=(X_i)_{i=1}^n\subset\mathcal{E}$ where $n\in\mathbb{N}$, let $U$ be the set on which all the $X_i$ are finite. Then $U$ is an open dense set and $f(X_1,\dots,X_n)$ is a well-defined continuous function on $U$. Then denote $f(\mathbf{X})$ to be the unique extension of $f(X_1,\dots,X_n)$ in $C^\infty(K)$. This enables us to prove the multivariate version of Jensen's inequality. The univariate version of the Jensen's inequality in the setting of vector lattices was proved in [Theorem 4.4, \cite{MR3151817}]. Let $\mathbf{X}=(X_1, \dots, X_n)$ and for a given conditional expectation operator $\mathbb{F}$ on $\mathcal{E}$ let $\mathbb{F}\mathbf{X}:=(\mathbb{F}X_1, \dots, \mathbb{F}X_n)$.
\begin{theorem}
Let $f\in C(\mathbb{R}^n, \mathbb{R})$ be a convex function and $\mathbf{X}=(X_i)_{i=1}^n\subset\mathcal{E}$. Let $\mathbb{F}$ be a conditional expectation defined on $\mathcal{E}$. If $f(\mathbf{X}) \in \mathcal{E}$, then $\mathbb{F}(f(\mathbf{X})) \geq f(\mathbb{F}\mathbf{X})$.
\end{theorem}
\begin{proof}
Since $f$ is a convex function, it is a fact from analysis that there exists a sequence of affine functions $L_m:\mathbb{R}^n\to\mathbb{R}$ of the form $L_m(t)=\langle a_m,t\rangle+b_m$ for some $a_m\in\mathbb{R}^n, b_m\in\mathbb{R}$ such that for every $t\in\mathbb{R}^n$, we have
$$f(t)=\sup_{m\in\mathbb{N}} L_m(t)$$
Since $f\geq L_m$, we have $f(\mathbf{X})\geq L_m(\mathbf{X})$ for every $m$. Since $\mathbb{F}$ is a positive linear projection and $f(\mathbf{X})\in\mathcal{E}$, we have
\begin{equation}\label{jensen_midway}
\mathbb{F}(f(\mathbf{X}))\geq \mathbb{F}(L_m(\mathbf{X}))=L_m(\mathbb{F}\mathbf{X}), \forall m\in\mathbb{N}
\end{equation}
For $m\in\mathbb{N}$, let $L_m^\prime=L_1\vee\dots\vee L_m$. Then $L_m^\prime$ are an increasing sequence of continuous functions such that $f(t)=\sup_{m\in\mathbb{N}}L_m^\prime(t)$ and thus $L_m^\prime(\mathbb{F}\mathbf{X})\uparrow f(\mathbb{F}\mathbf{X})$ by Theorem \ref{extension_properties}. By Corollary \ref{wedge_vee_composition} and \ref{jensen_midway}, we have
$$L_m^\prime(\mathbb{F}\mathbf{X})= \bigvee_{i=1}^m L_i(\mathbb{F}\mathbf{X})\leq \mathbb{F}(f(\mathbf{X}))$$
Thus, it follows that $f(\mathbb{F}\mathbf{X})\leq \mathbb{F}(f(\mathbf{X}))$.
\end{proof}

\begin{corollary}
Let $(X_{t}^{(i)}, \mathbb{F}_t, \mathcal{F}_t)_{i=1}^n$ be a finite collection of martingales with the filtration $(\mathbb{F}_t)_{t\in\mathbb{N}}$ and let $g\in C(\mathbb{R}^n,\mathbb{R})$ be a convex function. If $g(\mathbf{X}_t)=g(X_{t}^{(1)},\dots,X_{t}^{(n)})\in\mathcal{E}$ for all $t$, then $(g(\mathbf{X}_t),\mathbb{F}_t,\mathcal{F}_t)$ is a sub-martingale. 
\end{corollary}
\begin{proof}
It follows from Jensen's inequality that for $t<s$, we have
$$\mathbb{F}_t[g(\mathbf{X}_s)] \geq g[\mathbb{F}_t(\mathbf{X}_s)]=g(\mathbf{X}_t)$$
and thus $(g(\mathbf{X}_t),\mathbb{F}_t,\mathcal{F}_t)$ is a sub-martingale. 
\end{proof}

\section{Discrete stopping times}\label{discrete_stopping_times}
In the classical measure setting, given a filtered probability space $(\Omega, \mathcal{F}, (\mathcal{F}_{t})_{t\in T}, P)$, $(X_t)_{t\in T}$ an adapted stochastic process and $\tau$ a stopping time, we define the stopped process as $X_t^\tau(\omega) = X_{t\wedge\tau(\omega)}(\omega)$. The notion for discrete stopping times has been generalized to the theory of vector lattices by Kuo, Labuschagne and Watson in \cite{MR2093170} and further discussed in \cite{MR2118085}. Below we state the definitions for the stopping time and stopped processes in vector lattices. We will only consider the case of discrete filtrations and stochastic processes.

\begin{definition}
Let $(\mathbb{F}_i)_{i\in\mathbb{N}}$ be a filtration on $\mathcal{E}$. A stopping time $(P_i)_{i\in\mathbb{N}}$ is defined to be an increasing sequence of band projections such that $P_0=0$ and $\mathbb{F}_jP_i = P_i\mathbb{F}_j$ whenever $i\leq j$. 
\end{definition}
In particular, each $P_i$ is order continuous, $0\leq P_i\leq I$ and $\mathcal{R}(P_i)$ is a band in $\mathcal{E}$, hence an order complete sublattice. Note that $P_n E\in\mathcal{F}_n$, because $P_n E=P_n\mathbb{F}_n E=\mathbb{F}_n P_n E$.

The stopping time $(P_i)$ is said to be bounded if there exists $N$ so that $P_i = I$ for all $i\geq N$. For a bounded stopping time $P=(P_i)$ and an adapted stochastic process $(X_i, \mathbb{F}_i)$ define the stopped element as $X_P$ where $X_P := \sum_{i=1}^\infty (P_i - P_{i-1})X_i$. 

\begin{lemma}\label{lateral_sum}
Let $(U_n)_{n\in\mathbb{N}}$ be a sequence of pairwise disjoint clopen sets in $K$, and set $\tau = \sup_{n\in\mathbb{N}} n\mathbbm{1}_{U_n}$. Then $\tau\in C^\infty(K)_+$ with $\{\tau=n\}=U_n, \forall n\in\mathbb{N}$ and $\mathcal{R}(\tau)\subseteq\mathbb{N}\cup\{0\}\cup\{\infty\}$.
\end{lemma}
\begin{proof}
Since $(n\mathbbm{1}_{U_n})_{n\in\mathbb{N}}$ are pairwise disjoint functions in $C^\infty(K)_+$, we have that $\tau\in C^\infty(K)$. Since band projections are order continuous, $$\mathbbm{1}_{U_k}\tau=\mathbbm{1}_{U_k}(\sup_n n\mathbbm{1}_{U_n})=\sup_n n\mathbbm{1}_{U_n\cap U_k}=k\mathbbm{1}_{U_k}.$$
Hence, $U_k\subseteq \{\tau=k\}^\circ$ for all $k\in\mathbb{N}$. But let us suppose that $\exists n\in\mathbb{N},\exists \alpha\in V:=\{\tau=n\}^\circ$ such that $\alpha\notin U_n$. Then consider the function
$$T=\tau-\mathbbm{1}_{V\setminus U_n}$$
Clearly, $\tau=T$ on $(V\setminus U_n)^c$ and $T(V\setminus U_n)=n-1$ and thus $T<\tau$. However, $\mathbbm{1}_{U_k}T=\mathbbm{1}_{U_k}\tau-\mathbbm{1}_{(V\setminus U_n) \cap U_k}=k\mathbbm{1}_{U_k}-\mathbbm{1}_{(V\setminus U_n)\cap U_k}.$ Since $V\cap U_k=\emptyset$, we have $\mathbbm{1}_{(V\setminus U_n)\cap U_k}=0$ and thus $T\geq k\mathbbm{1}_{U_k}, \forall k\in\mathbb{N}$. But this is a contradiction, and thus, we have $\{\tau=n\}^\circ= U_n$.\\

The support of $\tau$ is $\overline{\bigcup_{i=1}^\infty U_i}$ and $\tau(U_i)\in\mathbb{N},\forall i\in\mathbb{N}$. Let $\omega\in \partial(\bigcup_{i=1}^\infty U_i)$. Then there exists a sequence $(\omega_n)\subset \bigcup_{i=1}^\infty U_i$ such that $\omega_n\to \omega$ and the tail of the sequence does not belong to any $U_i$. Therefore, WLOG, we can assume that $\omega_n\in \bigcup_{i=n}^\infty U_i$ and thus $\tau(\omega_n)\geq n$. Therefore, $\tau(\omega)=\infty$ and hence $\mathcal{R}(\tau) \subseteq \mathbb{N}\cup\{0\}\cup\{\infty\}$. However, since the range of $\tau$ is discrete, for $n\in\mathbb{N}$ we have $\{\tau=n\}=\{\tau<n-\frac{1}{2}\}\cup\{\tau>n+\frac{1}{2}\}$. Hence, $\{\tau=n\}$ is an open set and we have $\{\tau=n\}= U_n$.
\end{proof}

Maeda-Ogasawara theorem allows to represent the stopping time in terms of continuous functions on the Stone space of $\mathcal{E}$. Since the $(P_n)$ are band projections, each of the $P_n$ correspond to multiplication by a function of the form $\mathbbm{1}_{W_n}$ where $W_n$ is a clopen set, with the $(W_n)$ being an increasing sequence of clopen sets. Let $P_n^\prime = P_n - P_{n-1}$ for $n\geq 1$ and $P_0^\prime = 0$. Then each of the $P_n^\prime$ remains a band projection and corresponds to multiplication by a function of the form $\mathbbm{1}_{U_n}$ where $U_n:=W_n\setminus W_{n-1}$ are pairwise disjoint clopen sets. Let $V = K\setminus (\overline{\bigcup_{n=1}^\infty U_n})$ and set $\tau = \infty\mathbbm{1}_V+\sup_n n\mathbbm{1}_{U_n}$. Lemma \ref{lateral_sum} gives that $\sup_{n\in\mathbb{N}} n\mathbbm{1}_{U_n}\in C^\infty(K)$ and thus $\tau\in\mathcal{E}^s$ with $\mathcal{R}(\tau)\subseteq\mathbb{N}\cup\{\infty\}$. Furthermore, for $n\in\mathbb{N}$, we have $\{\tau=n\}=U_n$ is a clopen set and hence
$$\mathbbm{1}_{\{\tau=n\}}=P_n^\prime E=P_n E - P_{n-1}E \in \mathcal{F}_n$$
Since $\mathbbm{1}_{\{\tau=n\}}\in \mathcal{F}_n$ we also have $\mathbbm{1}_{\{\tau\leq n\}}\in \mathcal{F}_n$.\\

In fact, the converse is true as well. That is, if $\tau\in\mathcal{E}^s$ satisfying $\mathcal{R}(\tau) \subseteq \mathbb{N}\cup\infty$ and $\mathbbm{1}_{\{\tau=n\}}\in \mathcal{F}_n$ for all $n\in\mathbb{N}$, then $\tau$ corresponds to a stopping time. To see this, define the operator $P_n:\mathcal{E}\to\mathcal{E}$ via $P_n(f) = f\cdot\mathbbm{1}_{\{\tau\leq n\}}$ for $f\in \mathcal{E}$. Clearly, $(P_n)_{n\in\mathbb{N}}$ is an increasing sequence of band projections such that $P_0=0$. Then upon using the averaging property of conditional expectation operators, we obtain that
$$\mathbb{F}_n P_n(f) = \mathbb{F}_n f\cdot \mathbbm{1}_{\{\tau\leq n\}}=\mathbbm{1}_{\{\tau\leq n\}}\cdot \mathbb{F}_n(f) = P_n\mathbb{F}_n(f)$$
Hence, $\tau$ corresponds to a stopping time. We summarize the above in the theorem below. 
\begin{theorem}\label{stopping_time_representation}
Every stopping time corresponds to an element $\tau\in \mathcal{E}^s$ that satisfies $\mathcal{R}(\tau) \subseteq \mathbb{N}\cup\{\infty\}$ and $\mathbbm{1}_{\{\tau=n\}}\in \mathcal{F}_n$ for all $n\in\mathbb{N}$. The converse is also true.
\end{theorem}

Representing the stopping times as above is useful in proving their various properties. A variant of Lemma \ref{limit_stopping_time} given below has been proven in \cite{MR2861602} by Grobler in the case of continuous processes (i.e., when the index is $[0,\infty)$). However, the properties regarding the discrete stopping times can not be deduced directly from there. Hence, we explicitly derive the result using the representation of stopping times and stopped processes from Theorem \ref{stopping_time_representation}. 

\begin{lemma}
\label{limit_stopping_time}
The set of stopping times is closed under the following operations.
\begin{itemize}
    \item If $\sigma,\tau$ are stopping times then so are $\sigma\vee\tau,\sigma\wedge\tau$ and $\sigma+\tau$.
    \item If $(\tau_n)_{n\in\mathbb{N}}$ is a sequence of stopping times, then $\inf\tau_n$ and $\sup\tau_n$ are stopping times. This includes the case where $\tau_n$ is increasing (or decreasing) to the limit $\tau$.
\end{itemize}
\end{lemma}
\begin{proof}
If $\sigma,\tau$ are stopping times then $\{\sigma\vee\tau\leq n\}=\{\sigma\leq n\}\cap\{\tau\leq n\}$ and $ \{\sigma\wedge\tau\leq n\}=\{\sigma\leq n\}\cup\{\tau\leq n\}$ and thus we have 
$$\mathbbm{1}_{\{\sigma\vee\tau\leq n\}} = \mathbbm{1}_{\{\sigma\leq n\}}\wedge\mathbbm{1}_{\{\tau\leq n\}}\in\mathcal{F}_n.$$
$$\mathbbm{1}_{\{\sigma\wedge\tau\leq n\}} = \mathbbm{1}_{\{\sigma\leq n\}}\vee\mathbbm{1}_{\{\tau\leq n\}}\in\mathcal{F}_n.$$
Since, $\sigma\vee\tau,\sigma\wedge\tau\in\mathcal{E}^s$ with $\mathcal{R}(\sigma\vee\tau),\mathcal{R}(\sigma\wedge\tau)\subseteq \mathbb{N}\cup\{\infty\}$, we have that $\sigma\vee\tau$ and $\sigma\wedge\tau$ are stopping times. Also $\{\sigma+\tau\leq n\}=\bigcup_{0\leq s\leq n}\{\sigma=s\}\cap\{\tau=n-s\}$ and $\mathbbm{1}_{\{\sigma=s\}}, \mathbbm{1}_{\{\tau=n-s\}}\in\mathcal{F}_n$. Hence $\mathbbm{1}_{\{\sigma+\tau\leq n\}}\in \mathcal{F}_n$ and, $\sigma+\tau$ is a stopping time.\\

Let us denote $\sup\tau_n=\tau$. WLOG, $(\tau_n)$ is increasing; otherwise, replace $\tau_n$ with $\tau_1\vee\dots\vee\tau_n$. We first observe that since $\tau_n\in\mathcal{E}_+^s$, $\tau$ is well defined and contained in $\mathcal{E}_+^s$. We claim that $\mathcal{R}(\tau)\subseteq \mathbb{N}\cup\{\infty\}$. Suppose not. Then there exists a point $\omega\in K$ such that $\tau(\omega)=\alpha\notin\mathbb{N}\cup\{\infty\}$ and hence there exists a clopen set $V\subseteq K$ such that $\tau(V)\subseteq (\lfloor\alpha\rfloor, \lfloor\alpha\rfloor+1)$. Let $\sigma = \tau\cdot\mathbbm{1}_{V^c}+\lfloor\alpha\rfloor\mathbbm{1}_V$. Clearly, $\sigma\geq \tau_n$ for every $n\in\mathbb{N}$, which is a contradiction. Therefore, $\mathcal{R}(\tau)\subseteq \mathbb{N}\cup\{\infty\}$ and since $\tau$ is a continuous function taking discrete values, $\{\tau\leq k\}$ is a clopen set for $k\in\mathbb{N}$. By Lemma \ref{convergence_criterion}, there exists a co-meagre set $M\subseteq K$ such that $\tau_n(\omega)\uparrow \tau(\omega)$ for every $\omega\in M$. Then for a fixed $k\in\mathbb{N}$, we have 
$$\{\tau\leq k\}\cap M = \bigcap_n\{\tau_n\leq k\}\cap M$$
which implies that $\mathbbm{1}_{\{\tau\leq k\}}=\mathbbm{1}_{\bigl(\bigcap_n\{\tau_n\leq k\}\bigr)^\circ}$ on a co-meagre set. \\

We claim that $\mathbbm{1}_{\bigl(\bigcap_n\{\tau_n\leq k\}\bigr)^\circ}=\inf_n\mathbbm{1}_{\{\tau_n\leq k\}}$ in $\mathcal{E}$. Since, $\bigcap_n\{\tau_n\leq k\}\subseteq \{\tau_m\leq k\}$ for every $m$, we have $\mathbbm{1}_{\bigl(\bigcap_n\{\tau_n\leq k\}\bigr)^\circ}\leq \inf_n\mathbbm{1}_{\{\tau_n\leq k\}}$. On the other hand, let $\sigma:=\inf_n\mathbbm{1}_{\{\tau_n\leq k\}}$. Using a similar argument as above, we can show that $\mathcal{R}(\sigma)\subseteq \{0\}\cup\{1\}$ and therefore, $V:=\supp\sigma$ is a clopen set. Then $V\subseteq \{\tau_n\leq k\}, \forall n\in\mathbb{N}\implies V\subseteq \bigcap_n\{\tau_n\leq k\}\implies V\subseteq \biggl(\bigcap_n\{\tau_n\leq k\}\biggr)^\circ$. Hence, $\sigma\leq \mathbbm{1}_{\bigl(\bigcap_n\{\tau_n\leq k\}\bigr)^\circ}$ which establishes the claim. Hence, by the Baire category theorem, $\mathbbm{1}_{\{\tau\leq k\}}=\inf_n\mathbbm{1}_{\{\tau_n\leq k\}} \in\mathcal{F}_k$ and thus $\sup_n\tau_n$ is a stopping time. Similarly, we can argue for $\inf\tau_n$.
\end{proof}

\begin{remark}
Let $\tau$ be a stopping time and $(n_k)$ be an increasing sequence in $\mathbb{N}$. Then clearly there exists a positive, increasing continuous function $g:\mathbb{R}_+\cup\{\infty\}\to\mathbb{R}_+\cup\{\infty\}$ such that $g(k)=n_k, g(0)=0, g(\infty)=\infty$ and $g(t)\geq t$ for all $t$. Since $g$ is a continuous function, $g(\tau)$ is a well defined element of $\mathcal{E}^s_+$ with $\mathcal{R}(g(\tau)) = \mathbb{N}\cup\{\infty\}$. We claim that $g(\tau)$ is again a stopping time. Now, $\{g(\tau)= n_k\}=\{\tau=k\}$ and hence $\mathbbm{1}_{\{g(\tau)= n_k\}}=\mathbbm{1}_{\{\tau=k\}}\in\mathcal{F}_{k}\subseteq\mathcal{F}_{n_k}$. If $n\notin(n_k)_{k=1}^\infty$ then $\{g(\tau)=n\}=\emptyset$ and $\mathbbm{1}_{\{g(\tau)=n\}}=0\in\mathcal{F}_n$. Thus, $g(\tau)$ is a stopping time. We note that $g(\tau)$ is only determined by $\tau$ and $(n_k)$, and does not depend on the choice of $g$.
\end{remark}

For a given stopping time $\tau\in\mathcal{E}^s$ and $n\in\mathbb{N}$, the above theorem shows that $\tau\wedge n\mathbbm{1}$ is also a stopping time. So for an adapted process $(X_n)_{n\in\mathbb{N}}$, there exists a stopped element $X_{\tau\wedge n\mathbbm{1}}$ for every $n\in\mathbb{N}$. So we define the stopped process corresponding to $\tau$ and $(X_n)_{n\in\mathbb{N}}$ to be the sequence of stopped elements $(X_{\tau\wedge n\mathbbm{1}})_{n\in\mathbb{N}}$.\\

It is evident from above that a stopping time $\tau$ is bounded precisely when $\tau\in C(K)$. The above representation of stopping times enables us to do the same for the stopped processes. Let $\tau$ be a bounded stopping time. Then there exists $N\in\mathbb{N}$ such that $\mathcal{R}(\tau)\leq N$. Since $P_n - P_{n-1} = P_n^\prime$, given an adapted process $(X_n, \mathbb{F}_n)_{n\in\mathbb{N}}$ we have $P_n^\prime X_n = X_n.\mathbbm{1}_{\{\tau = n\}}$. Hence, $X_\tau = \sum_{n=1}^\infty X_n.\mathbbm{1}_{\{\tau = n\}}$. Moreover, as $\tau$ is a continuous function with values in $\mathbb{N}$, this implies that $\{\tau=n\}$ is a clopen set for every $n\leq N$ and $\{\tau=n\}=\emptyset$ when $n>N$. Therefore, $X_{\tau}=\sum_{n=1}^N X_n\cdot\mathbbm{1}_{\{\tau=n\}}$. Let $\omega\in K$, then evaluated pointwise, we have $X_\tau(\omega)=X_{\tau(\omega)}(\omega)$. Thus, the stopped process is $(X_{\tau\wedge n\mathbbm{1}})_{n\in\mathbb{N}}$ evaluated pointwise. 

\begin{theorem}
Let $(X_n, \mathbb{F}_n)_{n\in\mathbb{N}}$ be an adapted process on $\mathcal{E}$. Then, given a bounded discrete stopping time $\tau$, the stopped element $X_\tau\in \mathcal{E}$ is the function evaluated pointwise. That is, at  $\omega\in K$, the value of the stopped element is $X_{\tau(\omega)}(\omega)$.
\end{theorem}

The following proposition improves the result [Lemma 5.3, \cite{MR2093170}].
\begin{proposition}
Let $(X_n, \mathbb{F}_n)_{n\in\mathbb{N}}$ be an increasing adapted process. Then the following statements hold.
\begin{itemize}
    \item Let $\sigma$ and $\tau$ be two bounded stopping times. Then $X_{\sigma\vee\tau}=X_\sigma\vee X_\tau$ and $X_{\sigma\wedge\tau}=X_\sigma\wedge X_\tau$.
    \item Let $\tau_n$ be a sequence of bounded stopping times such that $\tau:=\sup_n\tau_n$ is bounded. Then $X_\tau=\sup_n X_{\tau_n}$.
    \item Let $\tau_n$ be a sequence of bounded stopping times. Then $X_{\inf_n \tau_n}=\inf_n X_{\tau_n}$.
\end{itemize}
\end{proposition}
\begin{proof}
By Lemma \ref{limit_stopping_time}, we have that $\sigma\vee\tau$ is a stopping time, and therefore, $X_{\sigma\vee\tau}$ is a continuous function in $\mathcal{E}$. Let $\omega\in K$, then $X_{(\sigma\vee\tau)(\omega)}(\omega)=X_{\sigma(\omega)\vee\tau(\omega)}(\omega)$. WLOG, $\tau(\omega)\geq \sigma(\omega)$. Then $X_{(\sigma\vee\tau)(\omega)}(\omega)=X_{\tau(\omega)}(\omega)$. Since $(X_n)$ is an increasing process, we have $X_{\tau(\omega)}\geq X_{\sigma(\omega)}$. Therefore, we have: $X_{(\sigma\vee\tau)(\omega)}(\omega)=X_{\tau(\omega)}(\omega)\vee X_{\sigma(\omega)}(\omega)$. Since the equality is valid for every point in $K$, we have $X_{\sigma\vee\tau}=X_\sigma\vee X_\tau$. Similarly, we can conclude for $X_{\sigma\wedge\tau}=X_\sigma\wedge X_\tau$.\\

WLOG, $(\tau_n)$ is increasing; otherwise, replace $\tau_n$ with $\tau_1\vee\dots\vee\tau_n$. Since $\tau_n$ and $\tau$ are bounded stopping times, $X_{\tau_n}$ and $X_\tau$ are well defined elements of $\mathcal{E}$. For an arbitrary $\omega\in K$, we have $X_{\tau(\omega)}(\omega)\geq X_{\tau_n(\omega)}(\omega)$ for every $n\in\mathbb{N}$. Therefore, $X_\tau\geq \sup X_{\tau_n}$. Let $\sigma\geq X_{\tau_n}$ for every $n\in\mathbb{N}$. By [Lemma 3.6, \cite{MR4366912}], there exists a co-meagre subset $D\subseteq K$ such that $\tau_n(\omega)\uparrow\tau(\omega)$ for every $\omega\in D$. Fix $\omega\in D$. Since $\tau$ is bounded, for large enough $n$ we have $\tau_n(\omega)=\tau(\omega)$. Hence, $\sigma(\omega)\geq X_{\tau_n(\omega)}(\omega)=X_{\tau(\omega)}(\omega)$. Therefore, by the Baire Category theorem $X_{\tau}=\sup_n X_{\tau_n}$. Similarly, we can prove for the infimum.
\end{proof}

The following theorem is a vector lattice version of the Début theorem \cite{MR2998762} for discrete stochastic processes. For any stopping time, there exists an adapted stochastic process and a subset of $\mathbb{R}$ such that the corresponding hitting time will be precisely this stopping time. The stochastic process can be chosen intuitively and similar to the classical probability case. It will be $1$ until just before the stopping time is reached, from which on, it will be $0$. The increasing process, therefore, first hits the set $\{1\}$ at the stopping time. 

\begin{theorem}\label{debut_theorem}
Let  $\tau$ be a discrete stopping time. Then there exists an adapted process $(X_n)_{n\in\mathbb{N}}$ such that $\tau(\omega)=\inf\{t\in\mathbb{N}: X_t(\omega)= 1\}$.
\end{theorem}
\begin{proof}
Let $S_n = \{\tau\leq n\}$ and let $X_n= \mathbbm{1}_{S_n}$ for all $n\in\mathbb{N}$. Since $\mathbbm{1}_{\{\tau\leq n\}}\in\mathcal{F}_n$, this implies that $X_n\in\mathcal{F}_n$ and thus, the stochastic process $(X_n)_{n\in\mathbb{N}}$ is increasing and adapted. Let $\sigma:K\to\overline{\mathbb{R}}$ defined via $\sigma(\omega)=\inf\{t\in\mathbb{N}: X_t(\omega)= 1\}$. Then $\{\sigma=n\}=\{X_{n-1}=0\}\cap \{X_n=1\}$. However, we also have $\{\tau=n\}=\{X_{n-1}=0\}\cap \{X_n=1\}$ for every $n\in\mathbb{N}$. Moreover, $\omega\in\{\tau=\infty\}\iff X_n(\omega)=0,\forall n\in\mathbb{N}\iff\sigma(\omega)=\inf\{\emptyset\}=\infty$. Hence, $\tau(\omega)=\inf\{t\in\mathbb{N}: X_t(\omega)= 1\}$.
\end{proof}

\begin{remark}
There is also the notion of stopping times for continuous stochastic processes in vector lattices, as discussed by Grobler in \cite{MR2741330, MR2861602, MR4192217}. Let $T\subset \mathbb{R}_+$ be an interval. When phrased in terms of the $C^\infty(K)$ representation, the stopping time for the filtration $(\mathbb{F}_t)_{t\in T}$ is an orthomorphism $\mathbb{S} \in \text{Orth}(\mathcal{E})_+$ such that $\mathbbm{1}_{\{\mathbb{S}\leq t\}^\circ} \in \mathfrak{F}_t$ for every $t\in T$. However, we can not emulate the above results regarding stopped processes for the continuous case. This is because an arbitrary union of nowhere dense sets will need not be nowhere dense. Moreover, we use the Baire category theorem at various points in this paper, which does not hold for arbitrary union of nowhere dense sets.
\end{remark}

\subsection*{Acknowledgment}
I would like to thank my advisor, Prof. Vladimir Troitsky, for his patient guidance, from the construction of the general idea and the argument to the details of writing this paper.

\bibliographystyle{alpha}
\bibliography{bibliography}

\end{document}